\documentclass[11pt, a4paper]{article}

\usepackage{amsmath, amssymb, amsthm, xspace, a4wide, amscd}
\usepackage[english]{babel}
\usepackage[T1]{fontenc}
\usepackage[latin1]{inputenc}

\newtheorem{thm}{Theorem}
\newtheorem{lem}{Lemma}

\newtheorem{ass}{Assumption A\hspace*{-4pt}}

\theoremstyle{definition}
\newtheorem{definition}{Definition}

\newcommand{\rr}{\ensuremath{\mathbb{R}}}
\newcommand{\zz}{\ensuremath{\mathbb{Z}}}

\begin{document}

\begin{center}\Large
Averaging principle for the heat equation driven by a general stochastic measure\footnote{The final version will be published in ''Statistics and Probability Letters''}
\end{center}

\begin{center}
Vadym Radchenko
\end{center}



\begin{abstract}
We study the one-dimensional stochastic heat equation in the mild form driven by a general stochastic measure $\mu$, for $\mu$ we assume only $\sigma$-additivity in probability. The time averaging of the equation is considered, uniform a.~s. convergence to the solution of the averaged equation is obtained.
\end{abstract}

\section{Introduction}
\label{scintr}

Averaging is an important tool for investigation of dynamical systems. It helps to describe main part of the behavior of solutions to equations.

The averaging principle for stochastic partial differential equations was investigated for different types of equations with different stochastic integrators. For example, two-time-scales system driven by two independent Wiener processes was considered in~\cite{fudu11},  the infinite-dimensional case was studied by~\cite{cerfre}, \cite{breh12}, \cite{wang12}. Weak convergence in the averaging scheme was investigated by \cite{cerr09}, \cite{fu17}.

\cite{peis17} considered the system with the slow component driven by a fractional Brownian motion, case with Poisson random measure was studied in~\cite{peip17}, with $\alpha$-stable noise -- in \cite{bao17}.

In these and other studies, stochastic processes should have finite moments or satisfy some regularity conditions. We will study the case of the more general integrator.

In this paper we consider convergence of the solutions of the one-dimensional stochastic heat equations, which can formally be written as
\begin{equation}
\label{eqhshf} 
du_{\varepsilon}(t, x)=\frac{\partial^2 u_{\varepsilon}(t, x)}{\partial x^2}dt +f(x, u_{\varepsilon}(t, x))dt+ \sigma(t/ \varepsilon,x)d\mu(x),\
u_{\varepsilon}(0, x)=u_0(x)\, .
\end{equation}
Here $\varepsilon>0$, $(t, x)\in [0, T]\times\rr$, and $\mu$ is a stochastic measure defined on Borel
$\sigma$-algebra on $\rr$. For $\mu$ we assume $\sigma$-additivity in probability only, assumptions for $f,\ \sigma$
and $u_0$ are given in Section~\ref{scprob}.

We will study convergence $u_{\varepsilon}(t, x)\to\bar{u}(t, x),\  \varepsilon\to 0$,
were $\bar{u}$ is the solution of the averaged equation
\begin{equation}
\label{eqhshfav} 
d\bar{u}(t, x)=\frac{\partial^2 \bar{u}(t, x)}{\partial x^2}dt +{f}(x, \bar{u}(t, x))dt+ \bar{\sigma}(x)\,d{\mu}(x) ,\
\bar{u}(0, x)=u_0(x)\, .
\end{equation}
We consider solutions to the formal equations~(\ref{eqhshf}) and (\ref{eqhshfav}) in the mild
form (see~(\ref{eqhsife}) and~(\ref{eqhsifb}) below), $\bar{\sigma}$ is defined in~(\ref{eqbarf}).

In comparison with the papers mentioned above, we do not have a special equation for the fast component and study the case of additive noise. The reason is that we cannot properly define the integral of random function with respect to $\mu$ in general case.

Existence, uniqueness, and some regularity properties
of mild solution of~(\ref{eqhshf}) were obtained in~\cite{rads09}. Generalization for the parabolic equation, driven by a stochastic measure was obtained in~\cite{bodumj}. Some other stochastic equations with such integrator were studied in \cite{radc14,radt15}.

The rest of this paper is organized as follows. In Section~\ref{scprel} we give the basic facts about stochastic measures. Section~\ref{scprob} contains the exact formulation of the problem and our assumptions. Formulation and proof of the main result are given in Section~\ref{scaver}.

\section{Preliminaries}
\label{scprel}

Let ${\sf L}_0={\sf L}_0(\Omega, {\mathcal F}, {\sf P} )$ be the set of all real-valued
random variables defined on the complete probability space $(\Omega, {\mathcal F}, {\sf P} )$. Convergence in ${\sf L}_0$ means the convergence in probability. Let ${\mathcal{B}}$ be a Borel $\sigma$-algebra on ${\mathbb R}$.

\begin{definition}
A $\sigma$-additive mapping $\mu:\ {\mathcal{B}}\to {\sf L}_0$ is called {\em stochastic measure} (SM).
\end{definition}

In other words, $\mu$ is a vector measure with values in ${\sf L}_0$. We do not assume additional measurability conditions, positivity or moment existence
for SM.

For a deterministic measurable function $g:\rr\to\rr$, $A\in{\mathcal{B}}$ and SM~$\mu$, an integral of the form $\int_{A}g\,d\mu$ is defined and studied in~\cite[Chapter 7]{kwawoy}. In particular, every bounded measurable $g$ is integrable with respect to (w.~r.~t.) any~$\mu$. An analog of the Lebesgue dominated convergence theorem holds for
this integral, see \cite[Proposition 7.1.1]{kwawoy}.

Examples of SMs are the following. Let $M_{t}$ be a square integrable martingale. Then $\mu(A)=\int_0^T {\bf 1}_{A}(t)\,dM_{t}$ is an SM.
If $W^H_{t}$ is a fractional Brownian motion with Hurst index~$H>1/2$ and $g: [0,
T]\to\mathbb{R}$ is a bounded measurable function then $\mu(A)=\int_0^T g(t){\bf 1}_{A}(t)\,dW^H_{t}$ is an SM, as follows from~\cite[Theorem~1.1]{memiva}.
An $\alpha$-stable random measure defined on ${\mathcal{B}}$ is an SM too, see \cite[Chapter 3]{samtaq}.

By $C$ and $C(\omega)$ we will denote a positive finite constants and positive random finite constants respectively whose exact values are not important.

We will use the following statement.

\begin{lem} \label{lmfkmu} (Lemma~3.1 in~\cite{rads09})
Let $\phi_l:\ {\rr}\to \rr,\ l\ge 1$ be measurable functions such that
$ \tilde{\phi}(x)=\sum_{l=1}^{\infty} \left|{\phi_l}(x)\right|$ is integrable w.r.t.~$\mu$. Then
$\sum_{l=1}^{\infty}\left(\int_{\rr} \phi_l\,d\mu \right)^2<\infty\ \mbox{\textrm ~a.~s.}$
\end{lem}

We will consider the Besov spaces $B^\alpha_{22}([c, d])$. Recall that the norm in this classical space for
$0<\alpha< 1$ may be introduced by
\begin{equation}
\label{eqbssn}
\|g\|_{B^\alpha_{22}([c, d])}=\|g\|_{L_{2}([c, d])}+\left(\int_0^{d-c}\ {\left(w_2(g, r)\right)^2}{r^{-2\alpha-1}}
\,dr\right)^{1/2},
\end{equation}
where
\[
w_2(g, r)=\sup_{0\le h\le r}\left(\int_{c}^{d-h} \left|g(y+h)-g(y)\right|^2\,dy\right)^{1/2}
\]
(see~\cite{kamont}). For any $j\in \zz$ and all $n\ge 0$, put
\[
d_{kn}^{(j)}=j+k 2^{-n},\quad 0\le k\le 2^n,\quad \Delta_{kn}^{(j)}=\left(d_{(k-1)n}^{(j)}, d_{kn}^{(j)}\right],\quad
1\le k\le 2^n\, .
\]

The following lemma is a key tool for estimates of the stochastic integral.

\begin{lem}\label{lmessih} (Lemma~3 in~\cite{radt15})
Let $Z$ be an arbitrary set, and function $q(z,s):
Z\times [j, j+1]\to\rr$ is such that all paths $q(z,\cdot)$ are continuous on $[j,j+1]$. Denote
$$
q_n(z,s)=q\big(z,j\big){\bf 1}_{\{j\}}(s)+\sum_{1\le k\le 2^n}q\big(z,{d}_{(k-1)n}^{(j)}\big){\bf 1}_{\Delta_{kn}^{(j)}}(s).
$$
Then the random function
$
\eta(z)=\int_{(j,j+1]}\,q(z, s)\,d\mu(s),\ z\in Z,
$
has a version
\begin{eqnarray*}
\widetilde{\eta}(z)=\int\limits_{(j,j+1]} \,q_0(z,s)\,d\mu(s)
 +\sum_{n\ge 1}
\Bigl(\int\limits_{(j,j+1]} \,q_n(z,s)\,d\mu(s)-\int\limits_{(j,j+1]}\,q_{n-1}(z,s)\,d\mu(s)\Bigr)
\end{eqnarray*}
such that for all $\beta>0$, $\omega\in\Omega$, $z\in Z$
\begin{eqnarray}
|\widetilde{\eta}(z)|\le|q(z,j)\mu((j,j+1])|\nonumber\\
 + \Bigl\{\sum_{n\ge 1}2^{n\beta}
\sum_{1\le k\le 2^{n}}|q(z,d_{kn}^{(j)})-q(z,d_{(k-1)n}^{(j)})|^2\Bigr\}^{1/2}\nonumber\\
\times\Bigl\{\sum_{n\ge 1}2^{-n\beta}\sum_{1\le k\le 2^{n}}|\mu(\Delta_{kn}^{(j)})|^2
\Bigr\}^{1/2}.\label{eqesqm}
\end{eqnarray}
\end{lem}
Theorem~1.1 of~\cite{kamont} implies that for $\alpha=(\beta+1)/2$,
\begin{equation}\label{eqesbs}
 \Bigl\{\sum_{n\ge 1}2^{n\beta}
\sum_{1\le k\le 2^{n}}|q(z,d_{kn})-q(z,d_{(k-1)n})|^2\Bigr\}^{1/2}\le C \|q(z, \cdot)\|_{B^\alpha_{22}([j, j+1])}.
\end{equation}

\section{The problem}\label{scprob}

Consider the heat equation~(\ref{eqhshf}) in the following mild sense
\begin{eqnarray}
\label{eqhsife} u_{\varepsilon}(t, x)=\int_{\rr}{p}(t, x-y)u_0(y)\,dy\nonumber\\
+\int_0^t ds \int_{\rr}{p}(t-s, x-y)f(y, u_{\varepsilon}(s, y))\,dy\nonumber\\
+\int_{\rr} d\mu(y) \int_0^t {p}(t-s, x-y)\sigma(s / \varepsilon, y)\,ds\, .
\end{eqnarray}

Here ${p}(t,  x) = (2\sqrt{\pi t})^{-1}\, e^{-\frac{x^2}{4t}}$
is the Gaussian heat kernel, $u_{\varepsilon}(t, x)=u_{\varepsilon}(t, x, \omega):[0, T]\times{\rr}\times\Omega\to\rr$ is an unknown measurable
random function, $\mu$ is a stochastic measure defined on Borel $\sigma-$algebra of $\rr$. For each $(t,x)$ equation~(\ref{eqhsife}) holds a.~s.

We make the following assumptions throughout the paper.

\begin{ass}\label{ass1} $u_0(y)=u_0(y, \omega):{\rr}\times\Omega\to\rr$ is measurable and bounded,
\[
|u_0(y)|\le C(\omega),\quad \left|u_0(y_1)-u_0(y_2)\right|\le C(\omega)\left|y_1-y_2\right|^{\beta(u_0)},\quad \beta(u_0)\ge 1/6\, .
\]
\end{ass}

\begin{ass}\label{ass3} $f(y, z):{\rr}^2\to\rr$ is measurable, bounded, and
\[
\left|f(y_1, z_1)-f(y_2, z_2)\right|\le L_f \left(\left|y_1-y_2\right| + \left|z_1-z_2\right|\right)\, .
\]
\end{ass}

\begin{ass}\label{ass5} $\sigma(s, y):{\rr_+}\times{\rr}\to\rr$ is measurable, bounded, and
\[
\left|\sigma(s, y_1)-\sigma(s, y_2)\right|\le L_{\sigma}\left|y_1-y_2\right|^{\beta(\sigma)},\quad 1/2<\beta(\sigma)<1.
\]
\end{ass}

\begin{ass}\label{ass6}
$|y|^{\tau}$ is integrable w.r.t. $\mu$ on $\rr$ for some $\tau>5/2$.
\end{ass}

Recall that for some $C,\ \lambda >0$
\begin{equation}\label{estpx}
\left|\frac{\partial p(t,x)}{\partial x}\right|\leq
\frac{C}{t}e^{-\frac{\lambda x^2}{t}}.
\end{equation}

Assume that there exist the following limit
\begin{equation}\label{eqbarf}
\bar{\sigma}(y)=\lim_{t\to\infty}\frac{1}{t}\int_0^t \sigma(s,y)\,ds.
\end{equation}
It is easy to see that $\bar{\sigma}(y)$ satisfies Assumption~A\ref{ass5}.

We will consider solution $\bar{u}$ to the equation
\begin{eqnarray}
\label{eqhsifb}
\bar{u}(t, x)=\int_{\rr}{p}(t, x-y)u_0(y)\,dy+\int_0^t ds \int_{\rr}{p}(t-s, x-y){f}(y, \bar{u}(s, y))\,dy\nonumber\\
+\int_{\rr} d\mu(y) \int_0^t {p}(t-s, x-y)\bar{\sigma}(y)\,ds\, ,
\end{eqnarray}
that is a mild form of~(\ref{eqhshfav}).

Theorem of~\cite{rads09} and Assumptions A\ref{ass1}--A\ref{ass6} give that solutions of~(\ref{eqhsife}) and~(\ref{eqhsifb}) exist, are unique, and have continuous in $(t,x)$ versions.

We also impose the following additional condition.

\begin{ass}\label{ass8}
Function $G_{\sigma}(r,y)=\int_0^{r} (\sigma(s,y)-\bar{\sigma}(y))\,ds,\ r\in\rr_+,\ y\in\rr$
is bounded.
\end{ass}

This holds, for example, if $\sigma(s,y)$ is periodic in $s$ for each fixed $y$, and the set of values of minimal periods is bounded.

\section{Averaging principle}\label{scaver}

\begin{lem}\label{lmpers}
Let $h(r, y)$ and $\bar{h}(y)$ be measurable and functions
\begin{eqnarray*}
H(r,y)=h(r, y)-\bar{h}(y),\ G(r,y)=\int_0^{r} (h(v, y)-\bar{h}(y))dv,\
r\in\rr_+,\ y\in\rr^d
\end{eqnarray*}
are bounded. Then
$$
\sup_{y\in\rr^d,D>0,\varepsilon>0,t\in (0,T]}\Bigl|\frac{1}{\sqrt{\varepsilon}}\int_0^t \frac{e^{-\frac{D}{t-s}}}{\sqrt{t-s}}\,(h(s/\varepsilon,
y)-\bar{h}(y))\,ds\Bigr|<+\infty.
$$
\end{lem}

\begin{proof} Using the substitution $v=(t-s)/\varepsilon$, we obtain
\begin{eqnarray*}
\int_0^t \frac{e^{-\frac{D}{t-s}}}{\sqrt{t-s}}\left({h}(s/\varepsilon, y)-\bar{h}(y)\right)\,ds
{=}
\sqrt{\varepsilon}\int_0^{t/\varepsilon}
\frac{e^{-\frac{D}{v\varepsilon}}}{\sqrt{v}}\,(h(t/\varepsilon-v, y)-\bar{h}(y))\,dv.
\end{eqnarray*}
Let $|h(r, y)-\bar{h}(y)|\le C_h$, and  we denote
\[
F_{\varepsilon}(r)=\int_0^{r} (h(t/\varepsilon-v, y)-\bar{h}(y))\,dv,\quad 0\le r\le t/\varepsilon.
\]
Then $F_{\varepsilon}(r)$ is bounded, $|F_{\varepsilon}(r)|\le C_F$, $C_F$ does not depend of~$\varepsilon$. We have
\begin{eqnarray*}
\Bigl|\int_0^{1}\frac{e^{-\frac{D}{v\varepsilon}}}{\sqrt{v}}\,(h(t/\varepsilon-v, y)-\bar{h}(y))\,dv\Bigr|\le C_h\int_0^{1}\frac{dv}{\sqrt{v}}=C,\\
\Bigl|\int_1^{t/\varepsilon}
\frac{e^{-\frac{D}{v\varepsilon}}}{\sqrt{v}}\,(h(t/\varepsilon-v, y)-\bar{h}(y))\,dv\Bigr|=\Bigl|\int_1^{t/\varepsilon}
\frac{e^{-\frac{D}{v\varepsilon}}}{\sqrt{v}}\,dF_{\varepsilon}(v)\Bigr|\\
=\Bigl|\frac{e^{-\frac{D}{v\varepsilon}}}{\sqrt{v}}F_{\varepsilon}(v)\Bigr|_1^{t/\varepsilon}
-\int_1^{t/\varepsilon}F_{\varepsilon}(v)\Bigl(\frac{De^{-\frac{D}{v\varepsilon}}}{\varepsilon\sqrt{v^{5}}}-
\frac{e^{-\frac{D}{v\varepsilon}}}{2\sqrt{v^{3}}}\Bigr)\,dv\Bigr|\\
\le 2C_F+C_F\int_1^{t/\varepsilon}\frac{De^{-\frac{D}{v\varepsilon}}}{\varepsilon\sqrt{v^{5}}}\,dv+C_F\int_1^{t/\varepsilon} \frac{1}{2\sqrt{v^{3}}}\,dv \stackrel{(*)}{\le}C,
\end{eqnarray*}
where in (*) we used that $\max_{D>0} De^{-\frac{D}{v\varepsilon}}={v\varepsilon}e^{-1}$.
\end{proof}

Note that function $H_{\sigma}(r,y)=\sigma(r,y)-\bar{\sigma}(y),\ r\in\rr_+,\ y\in\rr$
is bounded, as follows from A\ref{ass5}. The main result of our paper is the following.

\begin{thm} Assume than Assumptions A\ref{ass1}--A\ref{ass8} hold. Then there exist versions of  $u_{\varepsilon}$ and $\bar{u}$ such that for any $\gamma_1<\frac{1}{2}\Bigl(1-\frac{1}{2\beta(\sigma)}\Bigr)$
$$
\sup_{\varepsilon>0,t\in [0,T],x\in\rr} \varepsilon^{-\gamma_1}|u_{\varepsilon}(t, x)-\bar{u}(t, x)|<+\infty\ a.~s.
$$
\end{thm}

\begin{proof}
We take the versions of stochastic integrals defined by Lemma~\ref{lmessih} and continuous versions of $u_{\varepsilon}$ and $\bar{u}$.

\textbf{Step 1.} First, we estimate
\begin{eqnarray*}
\xi_{\varepsilon}=\int_{\rr} d\mu(y) \int_0^t {p}(t-s, x-y)\sigma(s/\varepsilon, y)ds \\
- \int_{\rr} d\mu(y) \int_0^t {p}(t-s, x-y)\bar{\sigma}(y)ds.
\end{eqnarray*}
We will show that
\begin{equation}\label{eqesxi}
|{\xi_\varepsilon}|\le C(\omega){\varepsilon}^{\gamma_1}\ \textrm{a.~s.}
\end{equation}

For fixed $z=(t,x)$ denote
$$
g(z,y)=\int_0^t {p}(t-s,x-y)(\sigma(s/\varepsilon, y)-\bar{\sigma}(y))\,ds.
$$
Thus,
\begin{equation}\label{eqxiva}
\xi_{\varepsilon}=\int_{\rr} g(z,y)\,d\mu(y)=\sum_{j\in\zz}\int_{(j, j+1]} g(z,y)\,d\mu(y),
\end{equation}
and we will estimate $\int_{(j, j+1]} g(z,y)\,d\mu(y)$ using the Besov space norm, inequality~(\ref{eqesbs}), and Lemma~\ref{lmessih}. Consider
\begin{eqnarray*}
g(z,y+h)-g(z,y)\\
 =C\int_0^t \frac{e^{-\frac{|x-y|^2}{4(t-s)}}}{\sqrt{t-s}}(\sigma(s/\varepsilon,
y+h)-\sigma(s/\varepsilon, y)-\bar{\sigma}(y+h)+\bar{\sigma}(y))\,ds\\
+ C\int_0^t \frac{\sigma(s/\varepsilon, y+h)-\bar{\sigma}(y+h)}{\sqrt{t-s}}(e^{-\frac{|x-y-h|^2}{4(t-s)}}-
e^{-\frac{|x-y|^2}{4(t-s)}})\,ds
=: I_1+I_2 .
\end{eqnarray*}
We have
\begin{eqnarray*}
\int_0^t \frac{e^{-\frac{|x-y|^2}{4(t-s)}}}{\sqrt{t-s}}(\sigma(s/\varepsilon,
y+h)-\sigma(s/\varepsilon, y))\,ds
\stackrel{{A\ref{ass5}}}{\le} L_{\sigma}h^{\beta(\sigma)}\int_0^t \frac{ds}{\sqrt{t-s}}=C h^{\beta(\sigma)}.
\end{eqnarray*}
By similar way we estimate terms with $\bar{\sigma}$ and obtain that $|I_1|\le C h^{\beta(\sigma)}$.

Note that for any $b>0$
\begin{equation}
\int_0^t\frac{1}{v} e^{-\frac{b}{v}} dv \stackrel{z=b/v}{=}\int_{b/t}^{\infty}\frac{e^{-z}}{z}dz
\le \int_{b/t}^1\frac{dz}{z}+\int_1^{\infty} e^{-z}dz
\le \left|\ln{\frac{T}{b}}\right|+1.\label{eqestr}
\end{equation}

Further, for $|x-y|\le 2$ we have
\begin{eqnarray*}
|I_2|\stackrel{{A\ref{ass5}}}{\le} C \int_0^t \frac{|e^{-\frac{|x-y-h|^2}{4(t-s)}}-
e^{-\frac{|x-y|^2}{4(t-s)}}|}{\sqrt{t-s}}\,ds
=C\int_0^t ds \Bigl|\int_{x-y-h}^{x-y}\frac{\partial p(t-s,r)}{\partial r}\,dr\Bigr|\\
\stackrel{(\ref{estpx})}{\le} C \int_0^t ds \int_{x-y-h}^{x-y}\frac{e^{-\frac{\lambda r^2}{t-s}}}{t-s}dr
\stackrel{v=t-s}{=} C \int_{x-y-h}^{x-y}dr \int_0^t \frac{e^{-\frac{\lambda r^2}{v}}}{v}dv\\
\stackrel{(\ref{eqestr})}{\le}
C\int_{x-y-h}^{x-y}\left(\left|\ln{\frac{T}{\lambda r^2}}\right|
+1\right)dr
\stackrel{|x-y|\le 2}{\le} C|h\ln h|.
\end{eqnarray*}
If $|x-y|> 2$ then $|x-y-h|> 1$, and
$$
|I_2|{\le} C \int_0^t ds \int_{x-y-h}^{x-y}\frac{e^{-\frac{\lambda r^2}{t-s}}}{t-s}dr\stackrel{r\ge 1}{\le} Ch.
$$
In both cases, $|I_2|\le Ch^{\beta(\sigma)}$ (recall that $\beta(\sigma)<1$). We arrive at
\begin{equation}\label{est2}
(w_2(g,r))^2\leq 2\sup\limits_{0\leq h\leq r}\int_{j}^{j+1-h}I_1^2dy+
2\sup\limits_{0\leq h\leq r}\int_{j}^{j+1-h}I_2^2dy\le Cr^{2\beta(\sigma)}.
\end{equation}

Also, we need to estimate $w_2(g,r)$ using the value of~$\varepsilon$.
From A\ref{ass5}, A\ref{ass8}, and Lemma~\ref{lmpers} it follows that
\begin{equation}\label{eqggce}
|g(z,y+h)|+|g(z,y)|\le C\sqrt{\varepsilon}.
\end{equation}
Therefore
\begin{equation}\label{est3}
(w_2(g,r))^2\le C\varepsilon.
\end{equation}

Multiply (\ref{est2}) in power $1-\theta$ and (\ref{est3})
in power $\theta$ for some $0<\theta<1$, and obtain
$
 (w_2(g,r))^2\leq Cr^{2\beta(\sigma)(1-\theta)}\varepsilon^{\theta}.
$
For finiteness of integral in (\ref{eqbssn}) we need
$$
2\beta(\sigma)(1-\theta)>2\alpha \Leftrightarrow \beta(\sigma)\theta<\beta(\sigma)-\alpha.
$$
For  $\alpha\rightarrow1/2+$, we get
$\theta\rightarrow \Bigl(1-\frac{1}{2\beta(\sigma)}\Bigr)-$.

We have $\gamma_1< 1/2$, therefore from (\ref{eqggce}) we obtain
\[
 |g(z,j)|\leq C\varepsilon^{\gamma_1},\qquad
 \|g(z,\cdot)\|_{L_2([j,j+1])}\leq C\varepsilon^{\gamma_1}.
\]

Thus, for any $\gamma_1<\frac{1}{2}\Bigl(1-\frac{1}{2\beta(\sigma)}\Bigr)$ exists $\alpha>1/2$ such that
\[
 \|g(z,\cdot)\|_{B_{22}^\alpha([j,j+1])}\leq C \varepsilon^{\gamma_1}.
\]

Using (\ref{eqxiva}), (\ref{eqesqm}), (\ref{eqesbs}), and the Cauchy--Schwarz inequality, get
\begin{eqnarray*}
 |\xi_{\varepsilon}|
\le \sum_{j\in\zz}\left|\int_{(j, j+1]} g(z,y)\,d\mu(y)\right|
 \leq\sum_{j\in\zz} \left|g(z,j)\mu\left((j, j+1]\right)\right|\\
+ C \sum_{j\in\zz} \left\|g(z,\cdot)\right\|_{B^\alpha_{22}\left([j, j+1]\right)}
\left\{\sum_{n\ge 1}2^{n(1-2\alpha)}
\sum_{1\le k\le 2^{n}}\left|\mu\left(\Delta_{kn}^{(j)}\right)\right|^2 \right\}^{1/2}\\
\leq C \varepsilon^{\gamma_1}\left[\sum_{j\in\zz} \left|\mu\left((j, j+1]\right)\right|
+\sum_{j\in\zz} \left\{\sum_{n\ge 1}2^{n(1-2\alpha)}\sum_{1\le k\le 2^{n}}
\left|\mu\left(\Delta_{kn}^{(j)}\right)\right|^2 \right\}^{1/2}\right] \\
\leq C \varepsilon^{\gamma_1}
\left[ \left(\sum_{j\in\mathbb{Z}}(|j|+1)^{\rho}\left(\mu\left((j, j+1]\right)\right)^2
\right)^{1/2}\left(\sum_{j\in\mathbb{Z}}(|j|+1)^{-\rho}\right)^{1/2}
\right.\\
+\left. \left(\sum_{j\in\zz}\sum_{n\ge 1} \sum_{1\le k\le 2^n}
(|j|+1)^{\rho} 2^{n(1-2\alpha)}
\left|\mu\left({\Delta_{kn}^{(j)}}\right)\right|^2\right)^{1/2}
\left(\sum_{j\in\mathbb{Z}}(|j|+1)^{-\rho}\right)^{1/2}\right].
\end{eqnarray*}

For all $\rho>1$ we have $\sum_{j\in\mathbb{Z}}(|j|+1)^{-\rho}<+\infty$, and sum with stochastic integrals are
 $\sum_{l=1}^{\infty}\left(\int_{\sf x} \phi_l\,d\mu \right)^2$, where
\begin{eqnarray*}
\left\{\phi_l (y),\ l\geq 1\right\}= \left\{\left(|j|+1\right)^{\rho/2}\, {\bf 1}_{(j, j+1]}(y),\
j\in \zz\right\},\\
\left\{\phi_l (y), l\ge 1\right\}= \Bigl\{\left(|j|+1\right)^{\rho/2}2^{n(1-2\alpha)/2}
{\bf 1}_{\Delta_{kn}^{(j)}}(y),j\in \zz, n\ge 1,
 1\le k\le 2^{n}\Bigr\}.
\end{eqnarray*}
From $\sum_{l=1}^\infty|\phi_l(y)|\le C(|y|^{\tau}+1)$
for $\tau=\rho/2$, using integrability of $|y|^\tau$, by Lemma~\ref{lmfkmu}
we arrive at~(\ref{eqesxi}).

\textbf{Step 2.} Using Step~1, A\ref{ass3}, and equality $\int_{\rr}{p}(s, y)\,dy=1$, obtain
\begin{eqnarray*}
|u_{\varepsilon}(t, x)- \bar{u}(t, x)|\\
\le \Bigl|\int_0^t ds \int_{\rr}{p}(t-s, x-y)(f(y,
u_{\varepsilon}(s, y))-f(y, \bar{u}(s, y)))\,dy\Bigr|
+ C(\omega) {\varepsilon}^{\gamma_1}\\
{\le} L_f \int_0^t ds \int_{\rr}{p}(t-s, x-y)|u_{\varepsilon}(s, y)-\bar{u}(s,
y)|\,dy+ C(\omega) {\varepsilon}^{\gamma_1}\\
\le  L_f \int_0^t\sup_{y\in\rr}|u_{\varepsilon}(s, y)-\bar{u}(s, y)|  ds \int_{\rr}{p}(t-s, x-y)\,dy+ C(\omega) {\varepsilon}^{\gamma_1}\\
=L_f \int_0^t\sup_{y\in\rr}|u_{\varepsilon}(s, y)-\bar{u}(s, y)| \,ds+ C(\omega) {\varepsilon}^{\gamma_1}.
\end{eqnarray*}

Thus, we get
\[
\sup_{x\in\rr}|u_{\varepsilon}(t, x)- \bar{u}(t, x)|\le L_f \int_0^t\sup_{y\in\rr}|u_{\varepsilon}(s, y)-\bar{u}(s, y)| \,ds+C(\omega){\varepsilon}^{\gamma_1}.
\]
By Gronwall's inequality we obtain
$
\sup_{x\in\rr}|u_{\varepsilon}(t, x)- \bar{u}(t, x)|\le C(\omega){\varepsilon}^{\gamma_1},
$
that finishes the proof. \end{proof}

\textbf{Remark.} It was assumed that $\beta(\sigma)>\frac{1}{2}$, therefore we can choose $\gamma_1>0$. For smooth $\sigma$ we can take any $\gamma_1<\frac{1}{4}$. The order of convergence equal to $\frac{1}{4}$ was obtained in~\cite{peip17} for the system driven by Brownian motion and Poisson random measure. Convergence rate equal to $\frac{1}{2}-$ was achieved in \cite{breh12} and \cite{wang12} for systems driven by Brownian motions only.

\section*{Acknowledgments}

This work was supported by Alexander von Humboldt Foundation, grant 1074615.
The author is grateful to Prof. M.~Z\"{a}hle for fruitful discussions during the preparation of this paper and thanks the Friedrich-Schiller-University of Jena for its hospitality.

\bigskip

\textsc{Department of Mathematical Analysis, Taras Shevchenko National University of Kyiv, Kyiv 01601, Ukraine}\\
\emph{E-mail adddress}: \verb"vradchenko@univ.kiev.ua"

\end{document}